\documentclass[final,3p]{elsarticle}
\usepackage{amsmath,amssymb}
\usepackage[all]{xy}
\usepackage{latexsym}
\usepackage{amsthm,a4wide,color}
\usepackage{amsmath,amscd,verbatim}
\usepackage{hyperref}
\usepackage{graphicx}

\theoremstyle{plain}
  \newtheorem{thm}{Theorem}[section]
  \newtheorem{lem}[thm]{Lemma}
  \newtheorem{prop}[thm]{Proposition}
  \newtheorem{cor}[thm]{Corollary}
\theoremstyle{definition}
  \newtheorem{defn}[thm]{Definition}
  
  \newtheorem{exmp}[thm]{Example}
  \newtheorem{rem}[thm]{Remark}
  \newtheorem{note}[thm]{Note}

\makeatletter \def\ps@pprintTitle{  \let\@oddhead\@empty  \let\@evenhead\@empty  \def\@oddfoot{\centerline{\thepage}} \let\@evenfoot\@oddfoot} \makeatother

\begin{document}

\newcommand{\oto}{{\lra\hspace*{-3.1ex}{\circ}\hspace*{1.9ex}}}

\newcommand{\lam}{\lambda}
\newcommand{\da}{\downarrow}
\newcommand{\Da}{\Downarrow\!}
\newcommand{\D}{\Delta}
\newcommand{\ua}{\uparrow}
\newcommand{\ra}{\rightarrow}
\newcommand{\la}{\leftarrow}
\newcommand{\lra}{\longrightarrow}
\newcommand{\lla}{\longleftarrow}
\newcommand{\rat}{\!\rightarrowtail\!}
\newcommand{\up}{\upsilon}
\newcommand{\Up}{\Upsilon}
\newcommand{\ep}{\epsilon}
\newcommand{\ga}{\gamma}
\newcommand{\Ga}{\Gamma}
\newcommand{\Lam}{\Lambda}
\newcommand{\CF}{{\cal F}}
\newcommand{\CG}{{\cal G}}
\newcommand{\CH}{{\cal H}}
\newcommand{\CN}{{\mathcal{N}}}
\newcommand{\CB}{{\cal B}}
\newcommand{\CT}{{\cal T}}
\newcommand{\CS}{{\cal S}}
\newcommand{\CV}{\mathfrak{V}}
\newcommand{\CU}{\mathfrak{U}}
\newcommand{\CP}{{\cal P}}
\newcommand{\CQ}{\mathcal{Q}}
\newcommand{\mq}{\mathcal{Q}}
\newcommand{\cu}{{\underline{\cup}}}
\newcommand{\ca}{{\underline{\cap}}}
\newcommand{\nb}{{\rm int}}
\newcommand{\Si}{\Sigma}
\newcommand{\si}{\sigma}
\newcommand{\Om}{\Omega}
\newcommand{\bm}{\bibitem}
\newcommand{\bv}{\bigvee}
\newcommand{\bw}{\bigwedge}
\newcommand{\dda}{\downdownarrows}
\newcommand{\dia}{\diamondsuit}
\newcommand{\y}{{\bf y}}
\newcommand{\colim}{{\rm colim}}
\newcommand{\fR}{R^{\!\forall}}
\newcommand{\eR}{R_{\!\exists}}
\newcommand{\dR}{R^{\!\da}}
\newcommand{\uR}{R_{\!\ua}}
\newcommand{\swa}{{\swarrow}}
\newcommand{\sea}{{\searrow}}
\newcommand{\bbA}{{\mathbb{A}}}
\newcommand{\frF}{{\mathfrak{F}}}
\newcommand{\frN}{{\mathfrak{N}}}
\newcommand{\id}{{\rm id}}
\newcommand{\cl}{{\rm cl}}
\newcommand{\sub}{{\rm sub}}

\numberwithin{equation}{section}
\renewcommand{\theequation}{\thesection.\arabic{equation}}

\begin{frontmatter}
\title{Sobriety of quantale-valued  cotopological spaces}

\author{Dexue Zhang} 
\address{School of Mathematics, Sichuan University, Chengdu 610064, China}

\begin{abstract}
For each commutative and integral quantale, making use of the fuzzy order between closed sets, a theory of sobriety for quantale-valued  cotopological spaces is established based on irreducible closed sets.
\end{abstract}

\begin{keyword}
 Fuzzy topology   \sep quantale \sep quantale-valued order   \sep quantale-valued cotopological space \sep sobriety \sep irreducible closed set
\end{keyword}

\end{frontmatter}

\section{Introduction}
A topological space $X$ is   sober if each of its irreducible closed subsets is the closure of exactly one point in $X$. Sobriety of topological spaces can be described via the well-known adjunction \[\mathcal{O}\dashv{\rm pt}\] between the category {\sf Top} of topological spaces and the opposite of the category {\sf Frm} of frames  \cite{Johnstone}. Precisely, $X$ is sober if $\eta_X\colon  X\lra {\rm pt}(\mathcal{O}(X))$ is a bijection (hence a homeomorphism), where $\eta$ denotes the unit of the adjunction $\mathcal{O}\dashv{\rm pt}$.

In the classical setting, a topological space can be described in terms of open sets as well as   closed sets, and we can switch between open sets and closed sets by taking complements. So, it makes no difference whether we choose to work with closed sets or with open sets.  In the fuzzy setting,  since the table of truth-values is usually a quantale, not a Boolean algebra,  there is no natural way to switch between open sets and closed sets. So,  it may make a difference whether we postulate topological spaces in terms of open sets or in terms of closed sets.  An example in this regard is exhibited in \cite{CZ09,CLZ11}.

The frame approach to sobriety of topological spaces makes use of open sets; while the irreducible-closed-set approach  makes use of closed sets. Extending the theory of sober spaces to the fuzzy setting is an interesting topic in fuzzy topology. Most of the existing works  focus on the frame approach; that is, to find a fuzzy counterpart of the category  {\sf Frm} of frames, then establish an adjunction between the category of fuzzy topological spaces and that of  \emph{fuzzy frames}. Works in this regard include Rodabaugh \cite{Rodabaugh}, Zhang and Liu \cite{ZL95}, Kotz\'{e} \cite{Kotze97,Kotze01}, Srivastava and Khastgir \cite{SK98}, Pultr and Rodabaugh \cite{PR03a,PR03b,PR08a,PR08b},   Guti\'{e}rrez Garc\'{i}a,   H\"{o}hle and de Prada Vicente \cite{GHP},  and Yao \cite{Yao11,Yao12}, etc. But, the irreducible-closed-set approach to sobriety of fuzzy topological spaces is seldom touched, except in Kotz\'{e} \cite{Kotze97,Kotze01}.

In this paper, making use of the fuzzy inclusion order between closed sets, we establish a theory of sobriety for quantale-valued  topological spaces based on irreducible closed sets. Actually, this theory concerns sobriety of \emph{quantale-valued cotopological spaces}. By a  quantale-valued cotopological space  we mean a ``fuzzy topological space" postulated in terms of closed sets (see Definition \ref{cotopology}). The term \emph{quantale-valued  topological space} is reserved for ``fuzzy topological space" postulated in terms of open sets (see Definition \ref{topology}).

It should be noted that in most works on fuzzy frames, the table of truth-values is assumed to be a complete Heyting algebra (or, a frame),  even a completely distributive lattice sometimes. But, in this paper, the table of truth-values  is only assumed to be a commutative and integral quantale. Complete Heyting algebra, BL-algebras and left continuous t-norms, are important examples of such quantales.

The contents are arranged as follows. Section 2 recalls basic ideas about   quantale-valued ordered sets  and   quantale-valued  $\CQ$-cotopological spaces. Section 3, making use of the quantale-valued order between closed sets in a $\CQ$-cotopological space,  establishes a theory of sober $\CQ$-cotopological spaces based on irreducible closed sets. In particular, the sobrification of a stratified $\CQ$-cotopological space is constructed. The last section, Section 4, presents some interesting examples in the case that $\CQ$ is the unit interval $[0,1]$ coupled with a (left) continuous t-norm.

\section{Quantale-valued ordered sets and quantale-valued cotopological spaces}
In this paper, $\CQ=(Q,\&)$ always denotes a commutative and integral quantale, unless otherwise specified. Precisely,  $Q$ is a complete lattice with a bottom element $0$ and a top element $1$, $\&$ is a binary operation on $Q$ such that $(Q,\&, 1)$ is a commutative monoid  and
$p\&\bv_{j\in J}q_j=\bv_{j\in J}p\& q_j$
for all $p\in Q$ and   $\{q_j\}_{j\in J}\subseteq Q$.

Since the semigroup operation $\&$ distributes over arbitrary joins, it  determines a binary operation $\ra$ on $Q$ via the adjoint property \begin{equation*} p\&q\leq r\iff q\leq p\ra r. \end{equation*}
The  binary operation $\ra$ is called the \emph{implication}, or the \emph{residuation}, corresponding to  $\&$.

Some basic properties of the binary operations $\&$ and $\ra$ are collected below, they can be found in many places, e.g.  \cite{Belo02,Rosenthal90}.

\begin{prop} Let $\mathcal{Q}$ be a quantale. Then
\begin{enumerate}[(1)]
\item  $1\ra p=p$.
\item  $p\leq q \iff 1= p\ra q$.
\item  $p\ra(q\ra r)=(p\& q)\ra r$.
\item  $p\&(p\ra q)\leq q$.
\item  $\Big(\bv_{j\in J}p_j\Big)\ra q=\bw_{j\in J}(p_j\ra q)$.
\item  $p\ra\Big(\bw_{j\in J} q_j\Big)=\bw_{j\in J}(p\ra q_j)$.
\end{enumerate}
\end{prop}
We often write $\neg p$ for $p\ra 0$ and call it the \emph{negation} of $p$. Though it is  true that $p\leq\neg\neg p$ for all $p\in Q$, the inequality $\neg\neg p\leq p$ does not always hold. A quantale $\mathcal{Q}$ is said to satisfy the {\it law of double negation} if  \[(p\ra 0)\ra 0 = p,\] i.e., $\neg\neg p= p$,  for all $p\in Q$.

\begin{prop}\cite{Belo02}\label{properties of negation} Suppose that $\CQ$ is a quantale that satisfies the law of double negation. Then \begin{enumerate}[(1)] \item $p\ra q = \neg(p\&\neg q)=\neg q\ra \neg p$. \item $p\&q =\neg (q\ra\neg p)=\neg (p\ra\neg q)$. \item $\neg(\bw_{i\in I}p_i) = \bv_{i\in I}\neg p_i$.  \end{enumerate}\end{prop}

In the class of quantales, the quantales with the unit interval $[0,1]$ as underlying lattice are of particular interest in fuzzy set theory \cite{Ha98,KMP00}. In this case, the semigroup operation $\&$ is called a left continuous t-norm on $[0,1]$ \cite{KMP00}. If a left continuous t-norm $\&$ on $[0,1]$ is a continuous function with respect to the usual topology, then  it is called a continuous t-norm.

\begin{exmp} (\cite{KMP00})
 Some basic   t-norms:
\begin{enumerate}[(1)]
\item The minimum t-norm: $a\&b=a\wedge b=\min\{a,b\}$. The corresponding implication is given by \[a\ra
b=\left\{\begin{array}{ll} 1, & a\leq b;\\
b, & a>b.\end{array}\right.\]
\item The product t-norm: $a\&b=a\cdot b$. The
corresponding implication  is given by $$a\ra
b=\left\{\begin{array}{ll} 1, & a\leq b;\\
b/a, & a>b.\end{array}\right.$$
\item The {\L}ukasiewicz t-norm:
$a\&b=\max\{a+b-1,0\}$. The corresponding implication is given by
$$a\ra b=
\min\{1, 1-a+b\}. $$ In this case, $([0,1],\&)$ satisfies the law of double negation. \item The nilpotent minimum t-norm: $$a\& b=\left\{\begin{array}{ll} 0, & a+b\leq 1;\\ \min\{a,b\}, & a+b>1.\end{array}\right.$$ The corresponding implication  is given by $$a\ra b=\left\{\begin{array}{ll} 1, & a\leq b;\\ \max\{1-a,b\}, & a>b.\end{array}\right.$$  In this case, $([0,1],\&)$ satisfies the law of double negation. \end{enumerate}\end{exmp}


A $\mathcal{Q}$-order (or an order valued in the quantale $\CQ$) \cite{Belo02,Wagner97} on a set $X$ is a reflexive and transitive $\mathcal{Q}$-relation on $X$. Explicitly, a $\mathcal{Q}$-order on  $X$  is a map $R\colon X\times X\lra Q$ such that $1= R(x,x)$   and  $R(y,z)\& R(x,y)\leq R(x,z)$ for all $x,y,z\in X$.  The pair $(X,R)$ is called a  $\mathcal{Q}$-ordered  set. As usual, we write $X$ for the pair $(X, R)$ and $X(x,y)$ for $R(x,y)$ if no confusion would arise.

If $R\colon X\times X\lra Q$ is a $\CQ$-order on $X$, then $R^{\rm op}\colon X\times X\lra Q$, given by $R^{\rm op}(x,y)=R(y,x)$, is also a $\CQ$-order on $X$ (by commutativity of $\&$), called the opposite of $R$.

A map $f\colon X\lra Y$ between $\mathcal{Q}$-ordered sets is $\mathcal{Q}$-order-preserving if $X(x_1,x_2)\leq Y(f(x_1),f(x_2))$  for all $x_1,x_2\in X$.
We write \[\mbox{$\CQ$-{\sf Ord}}\] for the category of $\CQ$-ordered sets and $\mathcal{Q}$-order-preserving maps.

Given a set $X$, a map $A\colon X\lra Q$ is called a  fuzzy set (valued in $\CQ$), the value $A(x)$ is  interpreted as the membership degree.
The map \[\sub_X\colon Q^X\times Q^X\lra Q,\] given by
\begin{equation*}\sub_X(A, B)=\bw_{x\in X}A(x)\ra B(x),\end{equation*} defines a $\mathcal{Q}$-order on $Q^X$. The value $\sub_X(A, B)$ measures the degree that $A$ is a subset of  $B$, thus, $\sub_X$ is called the \emph{fuzzy inclusion order} on $Q^X$ \cite{Belo02}.  In particular, if $X$ is a singleton set then the $\mathcal{Q}$-ordered set $(Q^X,\sub_X)$ reduces to the $\mathcal{Q}$-ordered set $(Q,d_L)$, where \[d_L(p,q)=p\ra q.\]

For any $p\in Q$ and   $A\in Q^X$, write $p\&A, p\ra A\in Q^X$ for the fuzzy sets given by $(p\&A)(x)=p\&A(x)$ and  $(p\ra A)(x)=p\ra A(x)$, respectively. It is easy to check that
\begin{equation*}p\leq \sub_X(A, B)\iff p\&A\leq B\iff A\leq p\ra B  \end{equation*} for all $p\in Q$ and $A, B\in Q^X$. In particular, $A\leq B$ if and only if $1= \sub_X(A, B)$.  Furthermore,
\begin{equation*}p\ra \sub_X(A, B)=\sub_X(p\&A, B)=\sub_X(A,p\ra B)\end{equation*} for all $p\in Q$ and $A, B\in Q^X$.

Given a map $f\colon X\lra Y$, as usual,  define $f^\ra\colon  Q^X\lra Q^Y$ and $f^\la\colon Q^Y\lra Q^X$   by \[f^\ra(A)(y)=\bv_{f(x)=y}A(x), \quad f^\la(B)(x)= B\circ f(x).\] The fuzzy set $f^\ra(A)$ is called the image of $A$ under $f$, and $f^\la(B)$   the preimage of $B$.

The following proposition is a special case of the enriched Kan extension in category theory \cite{Kelly,Lawvere73}. A direct verification is  easy and can be found in e.g. \cite{Belo02,LZ07}.
\begin{prop}For any map $f\colon X\lra Y$, \begin{enumerate}[(1)] \item  $f^\ra\colon  (Q^X,\sub_X)\lra (Q^Y,\sub_Y)$ is $\CQ$-order-preserving; \item  $f^\la\colon  (Q^Y,\sub_Y)\lra (Q^X,\sub_X)$ is $\CQ$-order-preserving; \item $f^\ra$ is left adjoint to $f^\la$, written $f^\ra\dashv f^\la$, in the sense that \[\sub_Y(f^\ra(A),B) = \sub_X(A,f^\la(B))\] for all $A\in Q^X$ and $B\in Q^Y$. \end{enumerate} \end{prop}

A fuzzy upper set \cite{LZ06} in a $\CQ$-ordered set $X$ is a map $\psi\colon X\lra Q$ such that \[ X(x,y)\&\psi(x) \leq\psi(y)\] for all $x,y\in X$. It is clear that $\psi\colon X\lra Q$ is  a fuzzy upper set if and only if $\psi\colon X\lra(Q, d_L)$ is $\CQ$-order-preserving.

Dually, a fuzzy lower set  in a $\CQ$-ordered set $X$ is a map $\phi\colon X\lra Q$ such that \[\phi(y)\&X(x,y) \leq\phi(x)\] for all $x,y\in X$. Or equivalently,  $\phi\colon X^{\rm op}\lra(Q, d_L)$ is a $\CQ$-order-preserving map, where $X^{\rm op}$ is the opposite of the $\CQ$-ordered set $X$.

\begin{defn} \label{irreducible} A fuzzy lower set $\phi$ in   a $\CQ$-ordered set $X$ is irreducible  if $\bv_{x\in X}\phi(x)=1$ and \[\sub_X(\phi, \phi_1\vee\phi_2) = \sub_X(\phi, \phi_1)\vee\sub_X(\phi, \phi_2)\] for all fuzzy lower sets $\phi_1,\phi_2$ in $X$. \end{defn}

Irreducible fuzzy lower sets are a counterpart of directed lower sets \cite{Gierz2003} in the quantale-valued setting. In particular, the condition  $\bv_{x\in X}\phi(x)=1$ is  a $\CQ$-version of the requirement that a directed set should be non-empty.

The following definition is taken from \cite{CZ09, Zhang07}.



\begin{defn}\label{cotopology}  A  $\CQ$-cotopology  on a set $X$ is a subset
$\tau$ of $Q^X$  subject to the following conditions:
\begin{enumerate}
\item[(C1)] $p_X\in\tau$ for all $p\in Q$;
\item[(C2)] $A\vee B\in\tau$ for all $A, B\in\tau$;
\item[(C3)]
$\bw_{j\in J} A_j\in\tau$ for each subfamily $\{A_j\}_{j\in J}$ of
$\tau$.\end{enumerate}The pair $(X,\tau)$ is called a $\CQ$-cotopological space; elements in $\tau$ are called closed sets of  $(X,\tau)$.  A  $\CQ$-cotopology $\tau$ is stratified if
\begin{enumerate}
\item[(C4)] $p\ra A\in\tau$ for all $p\in Q$ and $A\in\tau$.
\end{enumerate} A $\CQ$-cotopology $\tau$ is co-stratified if
\begin{enumerate}
\item[(C5)] $p\&A\in\tau$ for all $p\in Q$ and $A\in\tau$.
\end{enumerate}
A $\CQ$-cotopology $\tau$ is strong if it is both stratified and co-stratified.
\end{defn}

As usual, we often write $X$, instead of $(X,\tau)$, for a $\CQ$-cotopological space.

\begin{rem} If $\CQ$ is the quantale obtained by endowing $[0,1]$ with a continuous t-norm $T$, then $\tau\subseteq [0,1]^X$ is a strong $\CQ$-cotopology on $X$ if and only if $(X,\tau^*)$ is a fuzzy $T$-neighborhood space in the sense of Morsi \cite{Morsi95a}, where $\tau^*=\{1-A\mid A\in\tau\}$. In particular, if $\CQ$ is the quantale $([0,1],\min)$, then $\tau\subseteq [0,1]^X$ is a strong $\CQ$-cotopology on $X$ if and only if $(X,\tau^*)$ is a  fuzzy neighborhood space in the sense of Lowen \cite{Lowen1982}.  \end{rem}

A map $f\colon X\longrightarrow Y$ between
$\CQ$-cotopological spaces is continuous if
$f^\la(A)=A\circ f$  is closed in $X$ whenever $A$ is closed in $Y$.  We write
\[\mbox{$\CQ$-{\sf CTop}}\] for the category of $\CQ$-cotopological
spaces and continuous maps; and write
\[\mbox{{\sf S}$\CQ$-{\sf  CTop}}\] for  the
category of stratified $\CQ$-cotopological spaces and continuous maps. It is easily seen that both $\CQ$-{\sf CTop} and {\sf S}$\CQ$-{\sf  CTop} are well-fibred topological categories over {\sf Set} in the sense of \cite{AHS}.

Given a $\CQ$-cotopological space $(X,\tau)$, its closure operator
$^-\colon Q^X\lra Q^X$ is defined by
\[\overline{A}=\bw\{B \in\tau\mid A\leq B \}\] for all $A\in Q^X$. The closure operator of a $\CQ$-cotopological
space $(X,\tau)$ satisfies the following conditions: for all $A, B\in Q^X$,
\begin{enumerate}
\item[\rm (cl1)] $\overline{p_X}=p_X$ for all $p\in Q$;
\item[\rm (cl2)] $ \overline{A} \geq  A$;
\item[\rm (cl3)]
$\overline{A\vee B}=\overline{A}\vee\overline{B}$;
\item[\rm (cl4)] $\overline{\overline{A}}=\overline{A}$.
\end{enumerate}

\begin{prop} Let $X$ be a $\CQ$-cotopological space. The following are equivalent: \begin{enumerate}[\rm(1)]
\item $X$ is stratified. \item $p\&\overline{A}\leq \overline{p\&A}$ for all $p\in Q$ and $A\in Q^X$. \item
The closure operator $^-\colon (Q^X,\sub_X)\lra (Q^X,\sub_X)$ is $\CQ$-order-preserving. \end{enumerate}
\end{prop}
\begin{proof} $(1)\Rightarrow(2)$ Since $p\ra \overline{p\&A}$ is closed and $A\leq p\ra \overline{p\&A}$, it follows that $\overline{A} \leq p\ra \overline{p\&A}$, hence $p\&\overline{A}\leq \overline{p\&A}$.

$(2)\Rightarrow(3)$ For any $A,B\in Q^X$, let $p=\sub_X(A,B)=\bw_{x\in X}(A(x)\ra B(x))$. Since $p\&A\leq B$, then $p\&\overline{A}\leq \overline{p\&A}\leq \overline{B}$, hence $\sub_X(A,B)=p\leq \sub_X(\overline{A},\overline{B})$.

$(3)\Rightarrow(1)$ Let $A$ be a closed set and $p\in Q$. In order to see that $p\ra A$ is also closed, it suffices to check that $\overline{p\ra A}\leq p\ra A$, or equivalently, $p\leq \sub_X(\overline{p\ra A}, A)$. This is easy since $p\leq \sub_X(p\ra A,A)\leq \sub_X(\overline{p\ra A},\overline{A}) =\sub_X(\overline{p\ra A},A)$. \end{proof}

It follows immediately from item (3) in the above proposition that if $X$ is a stratified $\CQ$-cotopological space and if $B$ is a closed set  in $X$, then for all $A\in Q^X$, \[ \sub_X(\overline{A},B)\leq \sub_X(A,B)\leq\sub_X(\overline{A},\overline{B})  =\sub_X(\overline{A},B),\] hence \begin{equation*}
 \sub_X(A,B)=\sub_X(\overline{A},B).  \end{equation*}
This equation will be useful in this paper.

\begin{cor}In a stratified $\CQ$-cotopological space $X$, \[\overline{A}=\bw\big\{\sub_X(A,B)\ra B\mid B~\text{is closed in $X$}\big\}\] for all $A\in Q^X$. \end{cor}

Given a $\CQ$-cotopological space $(X,\tau)$, define $\Omega(\tau)\colon X\times X\lra Q$ by \[\Omega(\tau)(x,y) =\bw_{A\in\tau}(A(y)\ra A(x)).\] Then $\Omega(\tau)$ is a $\CQ$-order on $X$, called the \emph{specialization $\CQ$-order} of $(X,\tau)$ \cite{LZ06}. It is clear that each closed set in $(X,\tau)$ is a fuzzy lower set in the $\CQ$-ordered set $(X,\Omega(\tau))$.

As said before, we often write $X$, instead of $(X,\tau)$, for a $\CQ$-cotopological space. Accordingly, we will write $\Omega(X)$  for the $\CQ$-ordered set obtained by equipping $X$ with its specialization $\CQ$-order.

The correspondence $X\mapsto\Omega(X)$ defines a functor \[\Omega\colon  \CQ\text{-}{\sf CTop}\lra\CQ\text{-}{\sf Ord}.\] In particular, if $f\colon X\lra Y$ is a continuous map between $\CQ$-cotopological spaces, then $f\colon \Omega(X)\lra\Omega(Y)$ is $\CQ$-order-preserving, i.e., $\Omega(X)(x,y)\leq\Omega(Y)(f(x),f(y))$ for all $x,y\in X$.

Conversely, given a $\CQ$-ordered set $(X,R)$, the family $\Gamma(R)$ of fuzzy lower sets in $(X,R)$ forms a strong $\CQ$-cotopology on $X$, called the \emph{Alexandroff $\CQ$-cotopology} on $(X,R)$. The correspondence $(X,R)\mapsto\Gamma(X,R)=(X,\Gamma(R))$ defines a functor \[\Gamma\colon  \CQ\text{-}{\sf Ord}\lra \CQ\text{-}{\sf CTop} \] that is left adjoint to the functor $\Omega$ \cite{CZ09,LZ06}.

The following conclusion says that in  a stratified $\CQ$-cotopological space, the specialization $\CQ$-order is determined by closures of singletons, as in the classical case.

\begin{prop}\cite{Qiao} If $X$ is a stratified $\CQ$-cotopological space, then $\Omega(X)(x,y) = \overline{1_y}(x)$ for all $x,y\in X$. \end{prop}

\section{Sober $\CQ$-cotopological spaces}
Let $X$ be a topological space. A closed set $F$ in  $X$ is irreducible if it is non-empty and for any closed sets $A,B$ in $X$, $F\subseteq A\cup B$ implies either $F\subseteq A$ or $F\subseteq B$. A topological space is sober if every irreducible closed set in it is the closure of exactly one point. Sobriety is an interesting property in the realm of non-Hausdorff spaces  and it plays an important role in domain theory \cite{Gierz2003}.
In order to extend the theory of sober spaces to the fuzzy setting, the first step is to postulate  irreducible closed sets in a $\CQ$-cotopological space. Fortunately, this can be done in a natural way with the help of the fuzzy inclusion order between the closed sets in a $\CQ$-cotopological space.

\begin{defn}
A closed set $F$ in a $\CQ$-cotopological  space $X$ is irreducible if $\bv_{x\in X}F(x)=1$ and \[\sub_X(F,A\vee B) = \sub_X(F, A)\vee\sub_X(F,B)\] for all closed sets $A,B$ in $X$.\end{defn}

This definition is clearly an extension of that of irreducible closed sets in a topological space. We haste to emphasize that the fuzzy inclusion order, not the pointwise order, between closed sets is used here. The condition  $\bv_{x\in X}F(x)=1$ is   a $\CQ$-version of the requirement that $F$ is non-empty.

\begin{exmp}\begin{enumerate}[(1)]\item Let $X$ be a stratified $\CQ$-cotopological space. For any $x\in X$, the closure $\overline{1_x}$ of $1_x$ is irreducible. This follows from  that $\sub_X(\overline{1_x},A)= A(x)$ for any closed set $A$ in $X$.

\item A fuzzy lower set $\phi$ in a $\CQ$-ordered set $X$ is irreducible in the sense of Definition \ref{irreducible} if and only if $\phi$ is an irreducible closed set in the Alexandroff $\CQ$-cotopological space $\Gamma(X)$.\end{enumerate} \end{exmp}
Having the notion of irreducible closed sets (in a $\CQ$-cotopological space) at hand, we are now able to formulate the central notion of this paper.

\begin{defn}A  $\CQ$-cotopological space $X$ is sober if it is  stratified and each irreducible closed in $X$ is the closure of $1_x$ for a unique $x\in X$. \end{defn}

Write
\[\mbox{{\sf Sob}$\CQ$-{\sf  CTop}}\] for  the full
subcategory of {\sf S}$\CQ$-{\sf  CTop}  consisting of sober $\CQ$-cotopological spaces. This section concerns basic properties of sober $\CQ$-cotopological spaces. First,  we show  that the subcategory {\sf Sob}$\CQ$-{\sf  CTop} is reflective in {\sf S}$\CQ$-{\sf  CTop} and that the specialization $\CQ$-order of each sober  $\CQ$-cotopological space $X$ is  \emph{directed complete}  in the sense that every irreducible fuzzy lower set in  the $\CQ$-ordered set $\Omega(X)$ has a supremum.  
Then we will discuss the relationship between \begin{itemize}\setlength{\itemsep}{-2pt} \item sobriety and Hausdorff separation in a stratified $\CQ$-cotopological space; \item   sober topological spaces and sober $\CQ$-cotopological spaces   via the Lowen functor $\omega_\CQ$; \item sober $\CQ$-cotopological spaces and sober $\CQ$-topological spaces in the case that $\CQ$ satisfies the law of double negation. \end{itemize}

Given a stratified $\CQ$-cotopological space $X$, let \[\text{irr}(X)\] denote the set of all irreducible closed sets in $X$. For each closed set $A$ in $X$, define  \[s(A)\colon \text{irr}(X)\lra Q\] by \[s(A)(F) =\sub_X(F,A).\]
\begin{lem}Let $X$ be a stratified $\CQ$-cotopological space. \begin{enumerate}[(1)]\item $s(p_X)(F)=p$ for all $p\in Q$ and $F\in  {\rm irr}(X)$.
\item  $s(A)=s(B)\Leftrightarrow A=B$  for all closed sets $A,B$ in $X$.
\item  $s(A\vee B)=s(A)\vee s(B)$  for all closed sets $A,B$ in $X$.
\item  $s\big(\bw\limits_{j\in J}A_j\big) = \bw\limits_{j\in J}s(A_j)$  for each family $\{A_j\}_{i\in J}$ of closed sets in $X$.
\item  $s(p\ra A)= p\ra s(A)$  for all $p\in Q$ and all closed sets  $A$ in $X$.
\item  $\sub_X(A,B)=\sub_{{\rm irr}(X)}(s(A),s(B))$  for all closed sets $A,B$ in $X$. \end{enumerate}\end{lem}
\begin{proof} We check (6) for example. On one hand, \[\sub_X(A,B)   \leq \bw_{F\in \text{irr}(X)}( \sub_X(F,A) \ra \sub_X(F,B))  = \sub_{\text{irr}(X)}(s(A),s(B)).\] On the other hand,
\begin{align*} \sub_{\text{irr}(X)}(s(A),s(B))  & =\bw_{F\in \text{irr}(X)}(s(A)(F)\ra s(B)(F)) \\ & = \bw_{F\in \text{irr}(X)}(\sub_X(F,A)\ra \sub_X(F,B)) \\ & \leq \bw_{x\in X}(\sub_X(\overline{1_x}, A)\ra \sub_X(\overline{1_x},B)) \\ & = \bw_{x\in X}(A(x)\ra B(x)) \\ & = \sub_X(A,B). \end{align*} The  proof is finished. \end{proof}

By the above lemma, \[\{s(A)\mid A~\text{is a closed set of $X$}\}\] is a stratified $\CQ$-cotopology on $\text{irr}(X)$. We will write $s(X)$, instead of $\text{irr}(X)$, for the resulting $\CQ$-cotopological space.

\begin{prop}\label{s(X) is sober}  $s(X)$ is   sober for each stratified $\CQ$-cotopological space $X$. \end{prop}
\begin{proof}First of all, we note that for each irreducible closed set $F$ in $X$, the closure of $1_F$ in $s(X)$ is given by $s(F)$. So, it suffices to show that for each  closed set $A$ in $X$, if $s(A)$ is  irreducible in $s(X)$, then $A$ is irreducible in $X$. We prove the conclusion in two steps.

\textbf{Step 1}. $\bv_{x\in X}A(x)=1$.  Since $s(A)$ is an irreducible closed set in $s(X)$, it holds that \[\bv_{F\in s(X)}s(A)(F)= \bv_{F\in s(X)}\sub_X(F,A)=1.\]  Since for each $F\in s(X)$ and $x\in X$, \[F(x)\&\sub_X(F,A)=F(x)\& \bw_{z\in X}(F(z)\ra A(z))\leq A(x),\] it follows that \begin{align*}\bv_{x\in X}A(x)&\geq \bv_{x\in X}\bv_{F\in s(X)} F(x)  \& \sub_X(F,A) \\ & = \bv_{F\in s(X)}\bv_{x\in X} F(x)\& \sub_X(F,A)\\ &= \bv_{F\in s(X)}\sub_X(F,A)\\ &=1.\end{align*}

\textbf{Step 2}. $\sub_X(A,B\vee C)= (\sub_X(A,B))\vee(\sub_X(A,C))$ for all closed sets $B, C$ in $X$. Since $s(A)$ is irreducible in $s(X)$, \begin{align*}\sub_X(A,B\vee C)&= \sub_{s(X)}(s(A), s(B \vee  C)) \\&= \sub_{s(X)}(s(A),s(B)\vee s(C))\\ & = \sub_{s(X)}(s(A),s(B))\vee\sub_{s(X)}(s(A),s(C))\\ &= \sub_X(A,B)\vee\sub_X(A,C).\end{align*}

Therefore, $A$ is an irreducible closed set in $X$. \end{proof}

\begin{prop}For a stratified $\CQ$-cotopological space $X$, define \[\eta_X\colon  X\longrightarrow s(X)\] by  $\eta_X(x)=\overline{1_x}$. Then \begin{enumerate}[(1)] \item $\eta_X\colon  X\lra s(X)$ is   continuous.
\item $X$ is sober if and only if $\eta_X$ is a homeomorphism. \end{enumerate}    \end{prop}

\begin{proof}(1) For any closed set $A$ in $X$ and   $x\in X$, \[\eta_X^\la(s(A))(x) = s(A)(\eta_X(x))= A(x),\] hence $\eta_X^\la(s(A))=A$. This shows that $f$ is continuous.

(2) If $X$ is sober then each irreducible closed set in $X$ is of the form $\overline{1_x}=\eta_X(x)$ for a unique $x\in X$, hence $\eta_X\colon X\lra s(X)$ is a bijection. Since for each closed set $A$ in $X$ and $x\in X$, \[\eta_X^\ra(A)(\overline{1_x})=\bv\big\{A(z)\mid\eta_X(z)= \overline{1_x}\big\}=A(x)= s(A)(\overline{1_x}), \] it follows that $\eta_X^\ra(A)=s(A)$.  This shows that $\eta_X$ is a continuous closed bijection, hence a homeomorphism. The converse conclusion is trivial, since $s(X)$ is sober by Proposition \ref{s(X) is sober}. \end{proof}

\begin{thm}\label{soberification}
Let $f\colon X\lra Y$ be a continuous map between stratified $\CQ$-cotopological spaces. If $Y$ is sober, there is a unique continuous map $f^*\colon s(X)\longrightarrow Y$ such that ~$f=f^*\circ\eta_X$. 
\end{thm}

\begin{lem}Let $f\colon X\lra Y$ be a continuous map between stratified $\CQ$-cotopological spaces. Then for each irreducible closed set $F$  of $X$, the closure $\overline{f^\ra(F)}$ of the image of  $F$ under $f$ is an irreducible closed set of $Y$. \end{lem}
\begin{proof}For any closed sets $A,B$ in $Y$, \begin{align*}\sub_Y(\overline{f^\ra(F)},A\vee B)& =\sub_Y(f^\ra(F),A\vee B) & (A\vee B~\text{is closed})\\ &= \sub_X(F, f^\la(A\vee B))&(f^\ra\dashv f^\la)\\ &=\sub_X(F, f^\la(A))\vee \sub_X(F, f^\la(B)) &(F ~\text{is irreducible})\\ & = \sub_Y(f^\ra(F), A)\vee\sub_Y(f^\ra(F), B)\\ & = \sub_Y(\overline{f^\ra(F)}, A)\vee\sub_Y(\overline{f^\ra(F)},B),\end{align*} hence $\overline{f^\ra(F)}$   is  irreducible. \end{proof}

\begin{proof}[Proof of Theorem \ref{soberification}] \textbf{Existence}.  For each $F\in s(X)$, since $\overline{f^\ra(F)}$ is an irreducible closed set of $Y$ and $Y$ is sober, there is a unique $y\in  Y$ such that $\overline{f^\ra(F)}$ equals the closure of $1_y$. Define $f^*(F)$ to be this $y$. We claim that $f^*\colon s(X)\longrightarrow Y$ satisfies the conditions.

First, we show that $B\circ f^*$  is a closed set in $s(X)$ for any closed set $B$ in $Y$, hence $f^*$ is continuous. Let $A$ be the closed set $f^\la(B)=B\circ f$ in $X$. Then for any $F\in s(X)$, \begin{align*}s(A)(F) &= \sub_X(F, f^\la(B))= \sub_Y(f^\ra(F), B) \\ &= \sub_Y(\overline{f^\ra(F)}, B) = \sub_Y(\overline{1_{f^*(F)}}, B)\\ & = B(f^*(F))=B\circ f^*(F),\end{align*} thus, $B\circ f^*=s(A)$ and is closed in $s(X)$.

Second, for any $x\in X$, since \[1_{f(x)}\leq f^\ra(\overline{1_x})\leq \overline{f^\ra(1_x)} = \overline{1_{f(x)}},\] it follows that \[\overline{1_{f(x)}}=\overline{f^\ra(\overline{1_x})}= \overline{f^\ra(\eta_X(x))},\] hence $f^*(\eta_X(x))= f(x)$, showing that $f^*\circ\eta_X= f$.

Therefore, $f^*\colon s(X)\longrightarrow Y$ satisfies the conditions.

\textbf{Uniqueness}. Since $Y$ is sober, it suffices to show that if $g\colon  s(X)\lra Y$ is a continuous map such that $f=g\circ\eta_X$, then $\overline{1_{g(F)}} = \overline{f^\ra(F)}$ for all $F\in s(X)$.

Since for any $E\in s(X)$, \[\Omega(s(X))(E, F)= \overline{1_F}(E) =\sub_X(E,F), \] it follows that for any $x\in X$, \begin{align*}F(x) &= \sub_X(\overline{1_x}, F) \\ &= \Omega(s(X))(\eta_X(x), F)\\ &\leq \Omega(Y)(g(\eta_X(x)), g(F))\\ &=\Omega(Y)(f(x), g(F))\\ &= \overline{1_{g(F)}}(f(x)),  \end{align*} showing that $f^\ra(F)\leq \overline{1_{g(F)}}$, hence $\overline{f^\ra(F)}\leq \overline{1_{g(F)}}$.

Conversely, since $\overline{\eta_X^\ra(F)}$ is closed in $s(X)$, there is some closed set $A$ in $X$ such that $\overline{\eta_X^\ra(F)}=s(A)$. For any $x\in X$, since \[F(x)\leq \eta_X^\ra(F)(\eta_X(x))\leq s(A)(\eta_X(x))= A(x),\] it follows that $F\leq A$, then \[g^\ra(s(A)) = g^\ra(\overline{\eta_X^\ra(F)}) \leq \overline{g^\ra\circ\eta_X^\ra(F)} = \overline{f^\ra(F)}, \] hence \[\overline{f^\ra(F)}(g(F))\geq g^\ra(s(A))(g(F))\geq s(A)(F)=1.\] Therefore,  $\overline{1_{g(F)}}\leq \overline{f^\ra(F)}$. \end{proof}

Theorem \ref{soberification} shows that the full subcategory of sober $\CQ$-cotopological spaces is reflective in   {\sf S}$\CQ$-{\sf CTop}. For any $\CQ$-cotopological space $X$, the sober space $s(X)$ is called the \emph{sobrification} of $X$.

An important property of sober spaces is that the specialization order of a sober space is directed complete  \cite{Gierz2003,Johnstone}. The following Proposition \ref{sober is dcpo} says this is also true in the quantale-valued setting if we treat irreducible fuzzy lower sets as ``directed fuzzy lower sets".

\begin{defn}\cite{Wagner97,LZ07} A supremum of a fuzzy lower set $\phi$ in a $\CQ$-ordered set $X$ is an element $\sup\phi$ in $X$ such that \[X(\sup\phi, x)=\bw_{z\in X}(\phi(z)\ra X(z,x))=\sub_X(\phi,X(-,x))\] for all $x\in X$. \end{defn}

The notion of supremum of a fuzzy lower set in a $\CQ$-ordered set is a special case of that of \emph{weighted colimit}  in category theory \cite{Kelly}.

\begin{prop}\label{sober is dcpo} Let $X$ be a sober $\CQ$-cotopological space. Then each irreducible fuzzy lower set in the specialization $\CQ$-order of $X$ has a supremum. \end{prop}

\begin{proof}Let $\phi$ be an irreducible fuzzy lower set in the $\CQ$-ordered set $\Omega(X)$. First, we show that the closure $\overline{\phi}$ of $\phi$ in $X$ is an irreducible closed set. Let $A,B$ be closed sets  in $X$. By definition of the specialization $\CQ$-order, both $A$ and $B$ are fuzzy lower sets in $\Omega(X)$. Hence \[\sub_X(\overline{\phi},A\vee B) =\sub_X(\phi,A\vee B) = \sub_X(\phi, A)\vee \sub_X(\phi, B) = \sub_X(\overline{\phi}, A)\vee \sub_X(\overline{\phi}, B),\] showing that $\overline{\phi}$   is an irreducible closed set in $X$.  Since $X$ is sober, there is a unique $a\in X$ such that $\overline{\phi}=\overline{1_a}$. We claim that $a$ is a supremum of $\phi$ in $\Omega(X)$. That is, for all $x\in X$, \[\Omega(X)(a,x) = \bw_{z\in X}(\phi(z)\ra\Omega(X)(z,x)). \] On one hand, for each $z\in X$, since $\phi(z)\leq \overline{\phi}(z) = \overline{1_a}(z) =\Omega(X)(z,a)$, it follows that \[\phi(z)\ra\Omega(X)(z,x)\geq \Omega(X)(z,a)\ra\Omega(X)(z,x)\geq \Omega(X)(a,x),\] hence \[\Omega(X)(a,x) \leq \bw_{z\in X}(\phi(z)\ra\Omega(X)(z,x)). \] On the other hand, \begin{align*} \bw_{z\in X}(\phi(z)\ra\Omega(X)(z,x))&= \sub_X(\phi, \overline{1_x}) &(\overline{1_x}(z) = \Omega(X)(z,x)) \\ & =  \sub_X(\overline{\phi},\overline{1_x}) & (\text{$\overline{1_x}$ is closed})\\ &\leq \overline{\phi}(a)\ra\overline{1_x}(a) \\ &= \Omega(X)(a,x). & (\overline{\phi}(a)=1)\end{align*} This completes the proof. \end{proof}

It is well-known that a Hausdorff topological space is always sober \cite{Gierz2003,Johnstone}. The following proposition says this is also true for $\CQ$-cotopological spaces if $\CQ$ is linearly ordered.

A $\CQ$-cotopological space $X$ is Hausdorff if the diagonal $\Delta\colon  X\times X\lra Q$, given by \[\Delta(x,y)=\begin{cases}1, & x=y,\\ 0, & x\not= y, \end{cases}\] is a closed set in the product space $X\times X$.
\begin{prop}Let $\CQ=(Q,\&)$ be a linearly ordered quantale. Then  each stratified Hausdorff $\CQ$-cotopological space is sober. \end{prop}

\begin{proof} Let $X$ be a stratified Hausdorff $\CQ$-cotopological space. In order to see that $X$ is sober, it suffices to show that if $F$ is an irreducible closed set in $X$, then $F(x)\not=0$ for at most one point $x$ in $X$.
Suppose on the contrary that there exist different $x,y$ in $X$ such that $F(x)>0$ and $F(y)>0$. Let $b=\min\{F(x),F(y)\}$. Then $b>0$ by linearity of $\CQ$. Since $X$ is Hausdorff,  there exist  two families of closed sets in $X$, say, $\{A_j\}_{j\in J}$ and $\{B_j\}_{j\in J}$,  such that \[\Delta(x,y) = \bw_{j\in J} A_j(x)\vee B_j(y) .\]  Since $\Delta(x,y)=0$, there exists some $i\in J$ such that $A_i(x)\vee B_i(y)< b$.  Since $A_i(z)\vee B_i(z)=1$ for all $z\in X$, we have  either $F\leq A_i$ or $F\leq B_i$, hence either $F(x)<b$ or $F(y)<b$, a contradiction. \end{proof}

\begin{note}\label{T2 is not sober} The assumption that $\CQ$  is linearly ordered is indispensable in the above proposition. To see this, let $\CQ=\{0,a,b,1\}$ be the Boolean algebra with four elements; let $X$ be the discrete $\CQ$-cotopological space with two points $x$ and $y$. It is clear that $X$ is Hausdorff. One can verify by enumerating all possibilities that the map $\lam$  given by $\lam(x)=a$ and $\lam(y)=b$, is an irreducible closed set in $X$, but  it is neither the closure of $1_x$ nor that of $1_y$.  \end{note}

An element $a$ in a lattice $L$ is a coprime if for all $b,c\in L$,  $a\leq b\vee c$ implies that either $a\leq b$ or $a\leq c$ \cite{Johnstone}. A complete lattice $L$ is said to have enough coprimes if every element in $L$ can be written as the join of a family of coprimes. It is clear that every linearly ordered quantale has enough coprimes and the complete lattice of closed sets in a topological space has enough coprimes.

We say that an element in a quantale $\CQ=(Q,\&)$ is a coprime if it is a coprime in the underlying lattice $Q$;  and $\CQ$ has enough coprimes if the complete lattice $Q$ has enough coprimes.

It is easily seen that if $1\in Q$ is a coprime  and if $F$ is an irreducible closed set in a $\CQ$-cotopological space $X$,   then for any closed sets $A,B$ in $X$, $F\leq A\vee B$ implies either $F\leq A$ or $F\leq B$. Said differently,  in this case, an irreducible closed set in a $\CQ$-cotopological space  is a coprime in the lattice of its closed sets.

Let $\CQ$ be a quantale and $X$ be a (crisp) topological space. We say that a map $\lam\colon X\lra Q$ is upper semicontinuous if for all $p\in Q$, \[\lam_{[p]}=\{x\in X\mid \lam(x)\geq p\}\] is a closed set in $X$.

\begin{lem}Let $\CQ$ be a quantale with enough coprimes and $X$ be a  topological space. \begin{enumerate}[(1)] \item $\lam\colon X\lra Q$ is upper semicontinuous if and only if  $\lam_{[p]}$ is a closed set in $X$ for each coprime $p$ in $Q$. \item If both $\lam, \mu\colon X\lra Q$ are upper semicontinuous then so is $\lam\vee\mu$. \item The meet  of any family of upper semicontinuous maps is upper semicontinuous.  \item If $\lam\colon X\lra Q$ is upper semicontinuous then so is $p\ra\lam$ for all $p\in Q$. \end{enumerate}
\end{lem}

\begin{proof} (1) follows from the fact that $\CQ$ has enough coprimes and (2) is an immediate consequence of (1). The verification of (3) is straightforward. And (4) follows from   \[(p\ra\lam)_{[q]}=\{x\in X\mid q\leq p\ra\lam(x)\}= \{x\in X\mid p\&q\leq  \lam(x)\}= \lam_{[p\&q]}\] for all $p,q\in Q$.
 \end{proof}

The above lemma shows that if $\CQ$ is a quantale with enough coprimes, then for each topological space $X$, the family of upper semicontinuous maps $X\lra Q$ forms a stratified $\CQ$-cotopology on $X$. We write $\omega_\CQ(X)$ for the resulting stratified $\CQ$-cotopological space.

For each closed set $K$ in $X$, $1_K\colon X\lra Q$ is obviously  upper semicontinuous, hence every closed set in $X$ is also a closed in $\omega_\CQ(X)$. Moreover, for any $A\subseteq X$, the closure of $1_A$ in $\omega_\CQ(X)$  equals $1_{\overline{A}}$, where $\overline{A}$ is the closure of $A$ in $X$.

The correspondence $X\mapsto \omega_\CQ(X)$ defines an embedding functor \[\omega_\CQ\colon {\sf Top}\lra{\sf S}\CQ\text{-}{\sf CTop}.\]
This functor is one of the well-known Lowen functors in fuzzy topology \cite{Lowen1976}.

The following conclusion says that for a linearly ordered quantale $\CQ$, the notion of sobriety for $\CQ$-cotopological spaces is a \emph{good extension} in the sense of Lowen \cite{Lowen1976}.

\begin{prop}\label{good extension} If $\CQ$ is a linearly ordered quantale, then  a topological space $X$ is   sober if and only if the $\CQ$-cotopological space $\omega_\CQ(X)$ is   sober. \end{prop}

\begin{proof}\textbf{Necessity}. Let $\lam$ be an irreducible closed in $\omega_\CQ(X)$. Firstly, we show that for each $x\in X$, the value $\lam(x)$ is either $0$ or $1$. Suppose on the contrary that there is some $x\in X$ such that $\lam(x)$ is neither $0$ nor $1$. We proceed with two cases. If there is no element $y$ in $X$ such that $\lam(y)$ is strictly between $\lam(x)$ and $1$, let $\phi=1_{\lam_{[1]}}$ and $\psi$ be the constant map $X\lra Q$ with value $\lam(x)$. Then both $\phi$ and $\psi$ are closed in $\omega_\CQ(X)$ and $\lam\leq\phi\vee\psi$, but neither $\lam\leq\phi$ nor $\lam\leq\psi$, contradictory to that $\lam$ is irreducible. If there is some $y\in X$ such that $\lam(x)<\lam(y)<1$, let $\phi=1_{\lam_{[\lam(y)]}}$ and $\psi$ be the constant map $X\lra Q$ with value $\lam(y)$. Then both $\phi$ and $\psi$ are closed in $\omega_\CQ(X)$ and $\lam\leq\phi\vee\psi$, but  neither $\lam\leq\phi$ nor $\lam\leq\psi$, contradictory to that $\lam$ is irreducible. Therefore, $\lam=1_K$ for some closed set $K$ in $X$. Since $\lam$ is  irreducible in $\omega_\CQ(X)$, $K$ must be an irreducible closed set in $X$. Thus,  in the topological space $X$, $K$ is the closure of $\{x\}$ for a unique $x$, i.e., $K=\overline{\{x\}}$. This shows that in $\omega_\CQ(X)$,  $\lam$ is the closure of $1_x$ for a unique $x$. Therefore, $\omega_\CQ(X)$ is sober.

\textbf{Sufficiency}. We prove a bit more, that is, if $\CQ$  has enough coprimes and $\omega_\CQ(X)$ is sober, then $X$ is sober.

Let $K$ be an irreducible closed set in $X$. Firstly, we show that $1_K$ is an irreducible closed set in the $\CQ$-cotopological space $\omega_\CQ(X)$. That $1_K\colon X\lra Q$ is upper semicontinuous is trivial. For any closed sets $\lam,\mu$ in $\omega_\CQ(X)$, one has by definition that \[ \sub_X(1_K,\lam\vee\mu)= \bw_{x\in K}(\lam(x)\vee\mu(x)).\] For any coprime $p\leq \sub_X(1_K,\lam\vee\mu)$, $K$ is clearly a subset of $(\lam\vee\mu)_{[p]} = \lam_{[p]}\cup\mu_{[p]}$, hence  either $K\subseteq\lam_{[p]}$ or $K\subseteq\mu_{[p]}$, and then either $p\leq \sub_X(1_K, \lam)$ or $p\leq \sub_X(1_K, \mu)$. Therefore, \[\sub_X(1_K,\lam\vee\mu) \leq \sub_X(1_K, \lam)\vee\sub_X(1_K, \mu).\] The converse inequality \[\sub_X(1_K,\lam\vee\mu) \geq \sub_X(1_K, \lam)\vee\sub_X(1_K, \mu)\] is  trivial. Thus, $1_K$ is an irreducible closed set in $\omega_\CQ(X)$.

Since $\omega_\CQ(X)$ is sober, there is a unique $x\in X$ such that $1_K$ is the closure of $1_x$ in  $\omega_\CQ(X)$. Because  the closure of $1_x$ in $\omega_\CQ(X)$  equals $1_{\overline{\{x\}}}$, one gets $K=\overline{\{x\}}$.
 \end{proof}

\begin{note}The assumption in Proposition \ref{good extension} that $\CQ$  is linearly ordered is indispensable. To see this, let $\CQ=\{0,a,b,1\}$ be the Boolean algebra with four elements; let $X$ be the discrete  (hence sober) topological space with two points $x$ and $y$. It is clear that  $\omega_\CQ(X)$ is the discrete $\CQ$-cotopological space in Note \ref{T2 is not sober}, hence it is not sober.  \end{note}


At the end of this section, we discuss the relationship between sober $\CQ$-cotopological spaces and sober $\CQ$-topological spaces in the case that the quantale $\CQ$ satisfies the law of double negation.

\begin{defn} \label{topology} A $\CQ$-topology  on a set $X$ is a subset $\tau$ of   $Q^X$  subject to the following conditions: \begin{enumerate} \item[(O1)] $p_X\in\tau$ for all $p\in Q$; \item[(O2)] $U\wedge V\in\tau$ for all $U,V\in\tau$; \item[(O3)] $\bv_{j\in J}U_j\in\tau$ for each subfamily $\{U_j\}_{j\in J}$ of $\tau$.\end{enumerate}The pair $(X,\tau)$ is called a $\CQ$-topological space; elements in $\tau$ are called open sets of  $(X,\tau)$.\end{defn} A $\CQ$-topological space in the   above definition is also called a \emph{weakly stratified}  $\CQ$-topological space in the literature, see e.g. \cite{HS95,HS99}. A  $\CQ$-topology $\tau$ is stratified   \cite{HS99} if \begin{enumerate} \item[(O4)] $p\&U \in\tau$ for all $p\in Q$ and $U\in \tau$.\end{enumerate}

It is clear that if $\CQ=(Q,\&)$ is a frame, i.e., if $\&=\wedge$, then every $\CQ$-topology is stratified.


Let $\CQ$ be a quantale that satisfies the law of double negation. If $\tau$ is a (stratified) $\CQ$-cotopology on a set $X$, then  \[\neg(\tau)=\{\neg A \mid A\in\tau\}  \] is a (stratified) $\CQ$-topology on $X$, where $\neg A (x)=\neg(A(x))$ for all $x\in X$. Conversely, if $\tau$ is a (stratified) $\CQ$-topology on  $X$, then   \[\neg(\tau)=\{\neg A \mid A\in\tau\}  \] is a (stratified) $\CQ$-cotopology on $X$. So, for a quantale $\CQ$ that satisfies the law of double negation, we can switch freely between (stratified) $\CQ$-topologies and (stratified) $\CQ$-cotopologies, hence between open sets and closed sets.

If $\CQ=(Q,\&)$  satisfies the law of double negation, then for any $A,B\in Q^X$, \begin{equation*}\sub_X(A, B) 
= \sub_X(\neg B, \neg A) \end{equation*} and \begin{equation*}\bv_{x\in X}A(x)\&B(x)
=\neg\sub_X(A, \neg B) =\neg\sub_X(B,\neg A). \end{equation*}
These equations are clearly extensions of the properties listed in Proposition \ref{properties of negation}.

\begin{prop}\label{FR}Let $\CQ$ be a quantale that satisfies the law of double negation; and let $(X,\tau)$ be a stratified $\CQ$-cotopological space. Then for each irreducible closed set $F$ in $(X,\tau)$, the map   \[f_F\colon \neg(\tau)\lra Q, \quad f_F(U)=\bv_{x\in X}F(x)\&U(x) \] satisfies the following conditions: \begin{enumerate}[(Fr1)]\item $f_F(p_X)=p$. \item $f_F(U\wedge V)=f_F(U)\wedge f_F(V)$. \item $f_F\big(\bv_{i\in I}U_i\big) = \bv_{i\in I}f_F(U_i)$. \item $f_F(p\&U) = p\& f_F(U)$. \end{enumerate} Conversely, if $g\colon \neg(\tau)\lra Q$ is a map satisfying {\rm (Fr1)--(Fr4)}, then there is a unique irreducible closed set $F$ in $(X,\tau)$ such that $g=f_F$. \end{prop}

\begin{proof} We check (Fr2) for example.
\begin{align*}f_F(U\wedge V)&=\neg (\sub_X(U\wedge V,\neg F))\\ &=\neg (\sub_X(F,\neg (U\wedge V)))\\ &=\neg (\sub_X(F,\neg U\vee\neg V)) \\ &=\neg(\sub_X(F,\neg U)\vee\sub_X(F,\neg V))\\ &=\neg(\sub_X(U,\neg F))\wedge\neg(\sub_X(V,\neg F))\\ &  =f_F(U)\wedge f_F(V).\end{align*}

Conversely, suppose $g\colon \neg(\tau)\lra Q$ is a map that satisfies  (Fr1)--(Fr4). Let \[F=\bw\{A\in\tau\mid g(\neg A)=0\}.\] We show that $F$ is an irreducible closed set in $(X,\tau)$ and $g=f_F$.

\textbf{Step 1}.   $g(\neg F)=0$. This follows from (Fr3) and Proposition \ref{properties of negation}(3).

\textbf{Step 2}. $g(U) =\bv_{x\in X}F(x)\&U(x)$ for all $U\in\neg(\tau)$. On one hand, if we let $p=\sub_X(U,\neg F)$, then $p\&U\leq\neg F$, hence \[p\&g(U)=g(p\&U)\leq g(\neg F)=0.\] Therefore, \[g(U)\leq\neg \sub_X(U,\neg F)=\bv_{x\in X}F(x)\&U(x).\] On the other hand, since $g(\neg(g(U))\&U)= \neg(g(U))\&g(U)=0$, it follows that $\neg(g(U))\&U\leq \neg F$ by definition of $F$. Therefore, $\neg(g(U))\leq \sub_X(U,\neg F)$, hence \[g(U)\geq\neg \sub_X(U,\neg F)=\bv_{x\in X}F(x)\&U(x).\]

\textbf{Step 3}. $\bv_{x\in X}F(x)=1$. Otherwise, let $\bw_{x\in X}\neg F(x)=p$. Then $p\not=0$ and $p_X\leq\neg F$. Therefore, $g(\neg F)\geq g(p_X)=p$, contradictory to that $g(\neg F)=0$.

\textbf{Step 4}. $\sub_X(F,A\vee B)=\sub_X(F,A)\vee\sub_X(F,B)$ for all closed sets $A,B$ in $(X,\tau)$. In fact, \begin{align*}\sub_X(F,A\vee B)  &= \sub_X(\neg A\wedge\neg B, \neg F)\\ & =\neg(g(\neg A\wedge\neg B)) \\ & =\neg(g(\neg A))\vee \neg(g(\neg B))\\ &=\sub_X(\neg A, \neg F)\vee\sub_X(\neg B, \neg F)\\ &= \sub_X(F,A)\vee\sub_X(F,B). \end{align*}

The proof is completed. \end{proof}

For any stratified $\CQ$-topological space $(X,\tau)$ and  $x\in X$, the map \[f_x\colon \tau\lra Q, \quad f_x(U)=U(x)\] clearly satisfies (FR1)--(FR4) in Proposition \ref{FR}. This fact leads to the following:

\begin{defn}A stratified $\CQ$-topological space $(X,\tau)$ is sober if for each map $f\colon \tau\lra Q$ satisfying (FR1)--(FR4) in Proposition \ref{FR}, there is a unique $x\in X$  such that $f(U)=U(x)$ for all $U\in \tau$. \end{defn}

We leave it to the reader to check that if $\CQ=(Q,\&)$ is a frame, i.e., $\&=\wedge$, then the above definition of sober $\CQ$-topological spaces coincides with that in \cite{ZL95}.

\begin{prop} Let $\CQ$ be a quantale that satisfies the law of double negation. Then a $\CQ$-topological space $(X,\tau)$ is sober if and only if the $\CQ$-cotopological space $(X,\neg(\tau))$ is sober. \end{prop}

\begin{proof} \textbf{Necessity}: Let $F$ be an irreducible closed set in the $\CQ$-cotopological space $(X,\neg(\tau))$. By Proposition \ref{FR}, the map \[f_F\colon  \tau \lra Q, \quad f_F(U)=\bv_{x\in X}F(x)\&U(x) \] satisfies (Fr1)--(Fr4), hence there is a unique $a\in X$ such that $f_F(U)=U(a)$ for all $U\in\tau$. We claim that the closure of $1_a$ in  $(X,\neg(\tau))$ is $F$.  Since \[\neg F(a)= f_F(\neg F)= \bv_{x\in X}  F(x) \&(\neg F(x)) =0,\] then $F(a)=1$. Therefore, $\overline{1_a}\leq F$. Conversely, since \[\bv_{x\in X}F(x)\&\neg(\overline{1_a})(x)= f_F(\neg(\overline{1_a})) =\neg(\overline{1_a})(a)=0,\] it follows that \[F(x)\leq \neg(\neg(\overline{1_a})(x))= \overline{1_a}(x)\] for all $x\in X$, hence $F\leq \overline{1_a}$.

\textbf{Sufficiency}. Let $f\colon \tau\lra Q$ be a map satisfying (FR1)--(FR4). By Proposition \ref{FR}, there is an irreducible closed set $F$ in the $\CQ$-cotopological space $(X,\neg(\tau))$ such that $f(U) = \bv_{x\in X}F(x)\&U(x)$ for all $U\in\tau$.  Since $ (X,\neg(\tau))$ is sober, there is a unique $a\in X$ such that $\overline{1_a}=F$, Therefore, \begin{align*}f(U)&=\bv_{x\in X}F(x)\&U(x) \\ &= \neg\sub_X(F,\neg U) \\ &= \neg\sub_X(1_a, \neg U)&(\overline{1_a}=F, ~\neg U~\text{is  closed})\\ &=  U(a),\end{align*}
  completing the proof.  \end{proof}

\section{Examples}
This section discusses the sobriety of some natural $\CQ$-cotopological spaces in the case that $\CQ$ is  the unit interval $[0,1]$ endowed with a (left) continuous t-norm. In this section, we will  write $\underline{a}$, instead of $a_{[0,1]}$, for the constant map $[0,1]\lra[0,1]$ with value $a$.

Let $\CQ=([0,1],\&)$ with $\&$ being a (left) continuous t-norm   on $[0,1]$. We consider three $\CQ$-cotopologies on $[0,1]$: \begin{itemize}\setlength{\itemsep}{-2pt} \item $\tau_{C\&}$:   the stratified $\CQ$-cotopology   on $[0,1]$ generated by   $\{{\rm id}\}$ as a subbasis; \item $\tau_{S\&}$: the strong $\CQ$-cotopology on $[0,1]$ generated by $\{{\rm id}\}$ as a subbasis; \item $\tau_{A\&}$: the Alexandroff $\CQ$-cotopology  on the $\CQ$-ordered set  $([0,1], d_R)$, where $d_R(x,y)=y\ra x$. \end{itemize}

First of all, we list some facts about these $\CQ$-cotopologies.

\begin{enumerate}[(F1)] \item  A closed set in $([0,1],\tau_{A\&})$ is, by definition, a $\CQ$-order-preserving map    $\phi\colon ([0,1], d_L)\lra ([0,1], d_L)$, where $d_L(x,y)=x\ra y$. So, it is easy to verify that for all $\phi\in\tau_{A\&}$: \begin{itemize} \setlength{\itemsep}{0pt} \item $\phi$ is increasing, i.e., $\phi(x)\leq\phi(y)$ whenever $x\leq y$; \item $\phi(1)=1\iff \phi\geq{\rm id}$. \end{itemize}

\item For each $x\in[0,1]$, the closure of $1_x$ in $([0,1],\tau_{A\&})$ is $x\ra{\rm id}$, i.e., \[\overline{1_x}= x\ra{\rm id}.\]  On one hand, $x\ra{\rm id}$ is a closed set in $([0,1],\tau_{A\&})$ and $(x\ra{\rm id})(x)=1$, hence $\overline{1_x}\leq x\ra{\rm id}$. On the other hand,  for any $\phi\in\tau_{A\&}$, if $\phi(x)=1$, then   $\phi(t)=\phi(x)\ra\phi(t)\geq x\ra t$ for all $t\leq x$, hence $\phi\geq x\ra{\rm id}$.

\item Since every Alexandroff $\CQ$-cotopology is a strong $\CQ$-cotopology and ${\rm id}\in \tau_{A\&}$, it follows that
    \[\tau_{C\&}\subseteq \tau_{S\&}\subseteq \tau_{A\&}.\] Moreover, since $x\ra{\rm id}\in\tau_{C\&}$ for all $x\in X$,  the closure of $1_x$  in both $([0,1],\tau_{C\&})$ and $([0,1],\tau_{S\&})$ is $x\ra{\rm id}$.
 \item Since finite joins and arbitrary meets of right continuous maps $[0,1]\lra[0,1]$ are right continuous, and $x\ra{\rm id}$ is right continuous for all $x\in[0,1]$,  it follows that every closed set in $([0,1],\tau_{C\&})$, as a map from $[0,1]$ to itself,  is right continuous.
\item The   space  $([0,1],\tau_{C\&})$ is initially dense in the category of stratified $\CQ$-cotopological spaces. Indeed, for each stratified $\CQ$-cotopological space $(X,\tau)$, \[  \{(X,\tau)\stackrel{A}{\lra}([0,1],\tau_{C\&})\}_{A\in\tau}\] is an initial source in the topological category {\sf S}$\CQ$-{\sf  CTop}.
\end{enumerate}

\begin{prop}\label{Sierpinski is sober} Let $\CQ=([0,1],\&)$ with $\&$ being a left continuous t-norm on $[0,1]$. Then the stratified  $\CQ$-cotopological space $([0,1],\tau_{C\&})$ is sober. \end{prop}

\begin{proof}
It suffices to show that if $\phi$ is an irreducible closed set in $([0,1], \tau_{C\&})$, then $\phi= x\ra{\rm id}=\overline{1_x}$ for some $x\in [0,1]$.  Since $\phi$ is increasing and $\bv_{t\in[0,1]}\phi(t)=1$, one obtains that $\phi(1)=1$. Let \[x=\inf\{t\in[0,1]\mid \phi(t)=1\}.\] Since $\phi$ is right continuous by (F4), then $\phi(x)=1$, hence $\phi\geq\overline{1_x}=x\ra{\rm id}$.  We claim that $\phi=x\ra{\rm id}$. Otherwise, there is some $t<x$ such that $\phi(t)>x\ra t$. Since $x\ra t=\bw_{y<x}(y\ra t)$, there is  some $y\in(t,x)$ such that $\phi(t)>y\ra t$. It is clear that both $ y\ra{\rm id}$ and $\underline{\phi(y)}$ belong to $\tau_{C\&}$ and $\phi\leq (y\ra{\rm id})\vee\underline{\phi(y)}$, but neither $\phi\leq (y\ra{\rm id}) $ nor $\phi\leq  \underline{\phi(y)}$, contradictory to the assumption that $\phi$ is irreducible. \end{proof}

In the following we consider sobriety of the $\CQ$-cotopologies $\tau_{S\&}$ and $\tau_{A\&}$ in the case that $\&$ is the minimum t-norm, the product t-norm and the {\L}ukasiewicz t-norm.

\begin{prop} \label{alexandroff for product t-norm} Let $\CQ=([0,1],\&)$  with $\&$ being the product t-norm. Then the Alexandroff $\CQ$-cotopology $\tau_{AP}$ on  $([0,1], d_R)$ is not sober; but the strong $\CQ$-cotopology $\tau_{SP}$ on $[0,1]$ generated by  $\{{\rm id}\}$  is sober.   \end{prop}
\begin{proof} By   Example 3.11 in \cite{LZ06},   the Alexandroff $\CQ$-cotopology
$\tau_{AP}$ consists of maps $\phi\colon [0,1]\lra[0,1]$ subject to the following conditions:
\begin{enumerate}\item[(i)]   $\phi$ is increasing; and
\item[(ii)] $y/x \leq \phi(y)/{\phi(x)}$ whenever $x> y$, where we agree by convention that $0/0=1$.\end{enumerate}

We note that each $\phi$ in $\tau_{AP}$ is continuous on $(0,1]$, but it may be discontinuous at $0$.

It is easy to verify that the map $\phi\colon [0,1]\lra[0,1]$, given by $\phi(0)=0$ and $\phi(t)=1$ for all $t>0$, is an irreducible closed set in $([0,1],\tau_{AP})$, but it is not the closure of $1_x$ for any  $x\in[0,1]$. So,  $([0,1],\tau_{AP})$ is not sober.

In the following we prove in three steps that   $\tau_{SP}$ on $[0,1]$ is sober.

\textbf{Step 1}. We show that the stratified $\CQ$-cotopology $\tau_{CP}$  on $[0,1]$ generated by $\{{\rm id}\}$ is given by \[\tau_{CP}= \{\phi\wedge \underline{a} \mid a\in[0,1], \phi\in\CB\}, \] where, \[\CB=\{\phi\in\tau_{AP}  \mid   \phi\geq{\rm id}, \text{$\phi$ is continuous}\}.\]

It is routine to verify that $ \mathcal{C}=\{\phi\wedge \underline{a} \mid a\in[0,1], \phi\in\CB\}$ is a stratified $\CQ$-cotopology on $[0,1]$ that contains the identity ${\rm id}\colon [0,1]\lra[0,1]$, hence $\tau_{CP}\subseteq \mathcal{C}$.   To see that $\mathcal{C}$ is contained in $\tau_{CP}$,   it suffices to check that $\CB\subseteq\tau_{CP}$. Let $\phi\in\CB$. For each $x\in(0,1]$, define $g_x\colon [0,1]\lra[0,1]$ by \[g_x=\underline{\phi(x)}\vee((\phi(x)\ra x)\ra t)=\underline{\phi(x)}\vee\Big(\frac{x}{\phi(x)}\ra {\rm id}\Big).\]
Then $g_x\in \tau_{CP}$. We leave it to the reader to check that  \[\phi= \bw_{x\in(0,1]}g_x= \bw_{x\in(0,1]}\Big(\underline{\phi(x)}\vee \Big(\frac{x}{\phi(x)}\ra {\rm id}\Big)\Big),\] hence $\phi\in\tau_{CP}$. Therefore, $\mathcal{C}\subseteq \tau_{CP}$.

\textbf{Step 2}. We show that \[\tau_{SP}= \{\phi\in\tau_{AP} \mid \text{$\phi$ is continuous}\}.\] It is easily verified that $\mathcal{S}=\{\phi\in\tau_{AP} \mid \text{$\phi$ is continuous}\}$ is a strong $\CQ$-cotopology on $[0,1]$ that contains the identity ${\rm id}\colon [0,1]\lra[0,1]$, hence $\tau_{SP}\subseteq \mathcal{S}$. Conversely, for any $\phi\in\mathcal{S}$ with $\phi(1)>0$, let $\psi=\phi(1)\ra\phi = \phi/\phi(1)$. Then $\psi\in\tau_{AP}$, $\psi(1)=1$, and $\psi$ is continuous. Thus, $\psi\in\CB\subseteq \tau_{SP}$. Since $\tau_{SP}$ is strong and $\phi=\phi(1)\&\psi$, then $\phi\in\tau_{SP}$, therefore $\mathcal{S}\subseteq \tau_{SP}$.

\textbf{Step 3}.  $([0,1], \tau_{SP})$ is sober. Suppose $\phi$ is an irreducible closed set in $([0,1], \tau_{SP})$. Since $\phi$ is increasing and $\bv_{t\in[0,1]}\phi(t)=1$, one has $\phi(1)=1$, hence $\phi\in\CB\subset\tau_{CP}$.
Since $\tau_{CP}$ is coarser than $\tau_{SP}$, then $\phi$ is an irreducible closed set in the sober space $([0,1], \tau_{CP})$, and consequently, $\phi=x\ra {\rm id}$ for a unique $x\in [0,1]$. This shows that $\phi$ is the closure of $1_x$ for a unique $x\in[0,1]$ in $([0,1], \tau_{SP})$, hence  $([0,1], \tau_{SP})$ is sober.
  \end{proof}

\begin{prop}\label{alexandroff for Luka t-norm} Let $\CQ=([0,1],\&)$  with $\&$ being the {\L}ukasiewicz t-norm. Then the strong $\CQ$-cotopology $\tau_{SL}$ on $[0,1]$ generated by   $\{{\rm id}\}$  is sober and coincides with the Alexandroff $\CQ$-cotopology $\tau_{AL}$ on the $\CQ$-ordered set $([0,1], d_R)$. \end{prop}
\begin{proof}
By   Example 3.10 in \cite{LZ06}, the Alexandroff $\CQ$-cotopology $\tau_{AL}$ on  $([0,1], d_R)$ consists of maps $\phi\colon [0,1]\lra[0,1]$ that satisfy the following conditions:
\begin{enumerate}\item[(i)]   $\phi$ is increasing; and
\item[(ii)] $\phi$ is 1-Lipschitz, i.e., $\phi(x)-\phi(y)\leq x-y$ for all
$x\geq y$.  \end{enumerate}

Firstly, we show that the stratified $\CQ$-cotopology $\tau_{CL}$  on $[0,1]$ generated by $\{{\rm id}\}$ is given by  \[\tau_{CL} =\{\phi\wedge \underline{a} \mid a\in[0,1], \phi\in\tau_{AL}, \phi\geq{\rm id}\}. \]
It is routine to verify that $\mathcal{C} =\{\phi\wedge \underline{a} \mid a\in[0,1], \phi\in\tau_{AL}, \phi\geq{\rm id}\}$ is a stratified $\CQ$-cotopology on $[0,1]$ that contains the identity ${\rm id}\colon [0,1]\lra[0,1]$, hence $\tau_{CL}\subseteq \mathcal{C}$. To see that $\mathcal{C}$ is contained in $\tau_{CL}$, it suffices to check that for any $\phi\in\tau_{AL}$, if $\phi\geq{\rm id}$ then $\phi\in\tau_{CL}$. For each $x\in[0,1]$, define $g_x\colon [0,1]\lra[0,1]$ by \[g_x=\underline{\phi(x)}\vee((\phi(x)\ra x)\ra {\rm id}),\] i.e., \[g_x(t)= \max\{\phi(x),\min\{\phi(x)+ t-x, 1\}\}.\]
Then $g_x\in \tau_{CL}$. We leave it to the reader to check that  \[\phi= \bw_{x\in[0,1]}g_x= \bw_{x\in[0,1]} (\underline{\phi(x)}\vee((\phi(x)\ra x)\ra {\rm id})),\] hence $\phi\in\tau_{CL}$. Therefore, $\mathcal{C}\subseteq \tau_{CL}$.

Secondly, we show that  $\tau_{SL}= \tau_{AL}$.
For any $\phi\in\tau_{AL}$ with $\phi(1)>0$, it is clear that $\phi(1)\ra\phi \in\tau_{AL}$ and $(\phi(1)\ra\phi)(1)=1$. Thus, $\phi(1)\ra\phi\in\tau_{CL}\subseteq \tau_{SL}$. Since $\tau_{SL}$ is strong and $\phi=\phi(1)\&(\phi(1)\ra\phi)$, it follows that $\phi\in\tau_{SL}$, hence $\tau_{AL}\subseteq \tau_{SL}$. The converse inclusion $\tau_{SL}\subseteq \tau_{AL}$ is trivial.

Finally, we show that $([0,1], \tau_{SL})$ is sober.
Let $\phi$ be an irreducible closed set in $([0,1], \tau_{SL})$. Since $\phi$ is increasing and $\bv_{t\in[0,1]}\phi(t)=1$, then $\phi(1)=1$, hence $\phi\in \tau_{CL}$. Since $\tau_{CL}$ is coarser than $\tau_{SL}$, it follows that $\phi$ is an irreducible closed set in the sober space $([0,1], \tau_{CL})$, hence $\phi=x\ra {\rm id}$ for a unique $x\in [0,1]$. This shows that $\phi$ is the closure of $1_x$ for a unique $x\in[0,1]$ in $([0,1], \tau_{SL})$. Therefore, $([0,1], \tau_{SL})$ is sober.   \end{proof}

\begin{prop}Let $\CQ=([0,1],\&)$  with $\&$ being the  t-norm $\min$. Then the strong $\CQ$-cotopology $\tau_{SM}$ on $[0,1]$ generated by $\{{\rm id}\}$ is sober; but the Alexandroff $\CQ$-cotopology $\tau_{AM}$ on the $\CQ$-ordered set  $([0,1], d_R)$ is not sober.
 \end{prop}

\begin{proof} First of all, since $\&=\min$, every stratified $\CQ$-cotopological space is obviously a strong $\CQ$-cotopological space, then the strong $\CQ$-cotopology $\tau_{SM}$ on $[0,1]$ generated by $\{{\rm id}\}$ coincides with the stratified $\CQ$-cotopology $\tau_{CM}$ on $[0,1]$ generated by $\{{\rm id}\}$, hence it is sober by Proposition \ref{Sierpinski is sober}.

With the help of Example 3.12 in \cite{LZ06}, it can be verified that
\[\tau_{AM}=\{\phi\wedge \underline{a}\mid a\in[0,1], \text{$\phi\colon [0,1]\lra[0,1]$ is increasing}, \phi\geq{\rm id}  \}.\]
For any $a\in(0,1)$, the map $\phi\colon [0,1]\lra[0,1]$  given by \[\phi(t)=\begin{cases}1,& t>a,\\ t, & t\leq a, \end{cases}\] is an irreducible closed set in $([0,1],\tau_{AM})$, and it is not the closure of $1_x$ for any  $x\in[0,1]$, so, $([0,1],\tau_{AM})$ is not sober.
 \end{proof}
We would like to record here that  \[\tau_{CM}=\tau_{SM}=   \{\phi\in\tau_{AM}  \mid \text{$\phi$ is  right continuous}\}.\]
 It is easily verified that $\mathcal{C}= \{\phi\in\tau_{AM}  \mid \text{$\phi$ is  right continuous}\}$ is a stratified $\CQ$-cotopology on $[0,1]$ that contains the identity ${\rm id}\colon [0,1]\lra[0,1]$, hence $\tau_{CM}\subseteq \mathcal{C}$. To see the converse inclusion, we show firstly  that for all $\phi\in \mathcal{C}$, if $\phi(1)=1$ then $\phi\in\tau_{CM}$. For each $x\in[0,1]$, define $g_x\colon [0,1]\lra[0,1]$ by $g_x=\underline{\phi(x)}\vee(x\ra {\rm id})$. 
 Then $g_x\in \tau_{CM}$. Since \[\phi= \bw_{x\in[0,1]}g_x= \bw_{x\in[0,1]}(\underline{\phi(x)}\vee(x\ra {\rm id})),\] then $\phi\in\tau_{CM}$. Secondly, for any $\phi\in \mathcal{C}$, since $\phi(1)\ra\phi\in \mathcal{C}$, $(\phi(1)\ra\phi)(1)=1$, and $\phi=\underline{\phi(1)}\wedge(\phi(1)\ra\phi)$, it follows that $\phi\in \tau_{CM}$. Therefore, $\mathcal{C}\subseteq \tau_{CM}$.

\vskip 6pt

{\bf Acknowledgement} The author thanks  Dr. Hongliang Lai for the stimulating discussions during the preparation of this paper.

\end{document}